\documentclass[a4paper,10pt]{amsart}

\usepackage[utf8]{inputenc}
\usepackage[OT2,T1]{fontenc}
\usepackage[english]{babel}

\usepackage{amsmath}
\usepackage{amsfonts}
\usepackage{amsrefs}
\usepackage{amssymb}
\usepackage{amsthm}
\usepackage{mathrsfs}
\usepackage{listings}
\usepackage[mathcal]{euscript}

\usepackage{multirow}

\usepackage[dvipsnames]{xcolor}
\usepackage{graphicx}
\usepackage[all]{xy}
\usepackage{hyperref}

\usepackage{algorithm2e}

\usepackage{colortbl}
\definecolor{mygray}{gray}{0.92}

\usepackage{array}
\newcolumntype{C}[1]{>{\centering\arraybackslash$}p{#1}<{$}}

\usepackage{breqn}
\newcounter{myequation}[equation]

\usepackage{caption} 
\theoremstyle{plain}
\newtheorem{theorem}{Theorem}[section]

\newtheorem{question}[theorem]{Question}
\newtheorem{proposition}[theorem]{Proposition}
\newtheorem{lemma}[theorem]{Lemma}
\newtheorem{corollary}[theorem]{Corollary}

\theoremstyle{definition}

\theoremstyle{remark}
\newtheorem{remark}[theorem]{Remark}

\numberwithin{equation}{section}

\def\eps{\varepsilon}
\def\epsilon{\varepsilon}
\def\theta{\vartheta}
\def\tilde{\widetilde}

\DeclareMathOperator{\Aut}{Aut}

\DeclareMathOperator{\Gal}{Gal}
\DeclareMathOperator{\GL}{GL}

\DeclareMathOperator{\PSL}{PSL}

\DeclareMathOperator{\SL}{SL}

\def\C{\mathbb{C}}

\def\F{\mathbb{F}}

\def\N{\mathbb{N}}

\def\P{\mathbb{P}}
\def\Q{\mathbb{Q}}

\def\Z{\mathbb{Z}}

\usepackage{bm}

\def \PSLZ  {\PSL_2(\Z)}
\def \wt    {\kappa}

\author[Anni]{Samuele Anni}
\address{%
  Samuele Anni,
  Aix-Marseille Université, CNRS, Centrale Marseille, Institut de Mathématiques de Marseille
  case 907, 163 avenue de Luminy, F13288 Marseille cedex 9,
France}
\email{samuele.anni@univ-amu.fr}

\author[Assaf]{Eran Assaf}
\address{%
    Eran Assaf,
    Department of Mathematics, Dartmouth College, 6188 Kemeny Hall, Hanover, NH 03755, USA
}
\email{eran.assaf@dartmouth.edu}
\author[Lorenzo Garc\'ia]{Elisa Lorenzo Garc\'ia\vspace{-8ex}}
\address{%
	Elisa Lorenzo Garc\'ia,
  Universit\'e de Neuch\^atel, Institut de Math\'ematiques, Rue Emile-Argand 11, 2000, Neuch\^atel,
  Switzerland. ~--~
 Univ Rennes, CNRS, IRMAR - UMR 6625, F-35000
 Rennes, %
  France. %
}
\email{elisa.lorenzo@unine.ch, elisa.lorenzogarcia@univ-rennes1.fr}

\title{On smooth plane models for modular curves of Shimura type}

\begin{document}

\maketitle
\begin{abstract}
In this paper we prove that there are finitely many modular curves that admit a smooth plane model. Moreover, if the degree of the model is greater than or equal to $19$, no such curve exists. For modular curves of Shimura type we show that none can admit a smooth plane model of degree $5,6$ or $7$. Further, if a modular curve of Shimura type admits a smooth plane model of degree $8$ we show that it must be a twist of one of four curves.
\end{abstract}


\section{Introduction}

The compactification, by normalisation, of the quotient space of the complex upper half plane by the action of a subgroup $\Gamma$ of $\SL_{2}(\Z) = \Gamma(1)$, a modular group, is called a modular curve, $X_{\Gamma}$ and it admits the structure of a compact Riemann surface. Serre’s GAGA theorem tells us that $X_{\Gamma}(\C)$ is a projective complex algebraic curve. Furthermore, Shimura, in \cite{SHI71}*{Proposition 6.9}, proved that modular curves admit the structure of projective algebraic curves, see also \cite{DISH}*{\S 7.7}. The problem of computing equations for such curves and their projective embeddings has been a central topic in numerous papers, motivated by a plethora of applications. We will describe some of those while making a summary of the state of the art. 

Modular curves are moduli spaces for elliptic curves with a given level structure. Given a positive integer $N$ and a subgroup $G\subseteq \mathrm{GL}_2(\Z/N\Z)$, the modular curve $X_G$ parametrises pairs $(E,\phi)$ up to isomorphism, where $E$ is an elliptic curve and $\phi$ is a $G$-level structure on the $N$-torsion of $E$. Therefore, the explicit knowledge of models of modular curves becomes key for understanding properties of elliptic curves. This aspect leads to several different applications, for example towards coding theory or for solving Diophantine applications, starting with the proof of Fermat's last theorem where they appear in several steps.
In this article we study whether it is possible to have ``nice'' models, that is smooth plane models, for modular curves defined over the rationals.

In general, finding equations for modular curves of large level is computationally difficult, as it involves computing group actions on large spaces and linear algebra over large cyclotomic fields. However, for some groups it is easier to compute such curves, e.g.  in \cite{G1996} Galbraith computes modular curves for the groups $\Gamma_0(N)$. One can look for a slightly larger class of groups: given a positive integer $N$, a \emph{group of Shimura type} of level $N$, as introduced originally by Shimura in \cite{SHI71}, is a subgroup $\Gamma(H,t) \subseteq \PSL_{2}(\Z)$ projection of a subgroup $G(H,t)\subseteq \SL_{2}(\Z)$ of the form
$$
G(H,t) = \left \{ \left( \begin{array}{cc} a & b \\ 0 & d \end{array} \right)  \in GL_2(\Z / N \Z)  :  a \in H,  t \mid b \right \},
$$
where $H \subseteq (\Z / N \Z)^{\times}$ is a subgroup and $t \mid N$. We will call a modular curve of Shimura type any modular curve corresponding to the choice of a group of Shimura type.

The main result of this article is the following:

\begin{theorem}\label{thm:main}
There are finitely many modular curves which admit a smooth plane model over the rationals. There is no modular curve which admits a smooth plane model of degree greater or equal to $19$.
Moreover, there is no modular curve of Shimura type which admits a smooth plane model of degree $5, 6$ or $7$. A modular curve of Shimura type which admits a smooth plane model of degree $8$ must be a twist of one of four curves.  
\end{theorem}

The four curves mentioned in Theorem \ref{thm:main} are listed in \S 5,  \ref{genus21}.

Surprisingly, during the computations we found the example of a Galois trigonal, i.e. superelliptic of degree $3$, model of a modular curve of genus $6$. 

\subsection{State of the art}
Galbraith in \cite{G1996} presented several techniques to obtain explicit models of modular curves by computing projective embeddings, relying on the computation of spaces of modular forms. 
Let us recall Galbraith's approach for the modular curves $X_{0}(N)$ for a positive integer $N$. There is a well-known canonical affine equation for $X_{0}(N)$ using $N$-modular polynomials that are symmetric polynomials $\phi (x,y)\in \Z[x,y]$, of degree $N+1$ in each variable, such that $\phi(j(\tau),j(N\tau)) = 0$, where $j(\tau)$ be the classical modular $j$-function. These equations have very large degree, the model is highly singular, and the coefficients involved are enormous. Galbraith's approach consists in obtaining equations via the canonical embedding, which is suitable for practical computation since the differentials on the curve correspond to the weight $2$ cusp forms for $\Gamma_{0}(N)$. Chosen a basis $\{f_1, \dots , f_{g}\}$ for the weight $2$ cusp forms, the canonical map is translated into $\tau \mapsto [f_1(\tau): \dots : f_{g}(\tau)]$ and gives a map from $X_{0}(N)$ to $\P^{g-1}(\C)$ from the modularity of the forms $f_{i}(\tau)$. Galbraith's strategy is a key element in our approach towards the main theorem of this article.

Kohel in \cite{K1997} presented a different method which involves quaternions and a different approach towards the computation of the differentials. These approaches have been used, together with others, to collect the database of small modular curve models available in \verb+Magma+ \cite{Magma}.

Despite the lack of a general algorithm, models for several modular curves have been found in the literature with a wide range of applications in mind. We mention some of these, as well as their relevance towards various research directions. 

Baran found models for the isomorphic curves $X_{\mathrm{ns}^+}(13)$ and $X_{s^+}(13)$ in \cite{B2010}. The study of integral points on these curves relates to the Serre's uniformity question over $\Q$, as in \cite{S1972}. More recently, Dose, Mercuri and Stirpe \cite{DMS2017} proposed a new approach for computing (singular) models in order to study Serre's question.

Derickx, Najman and Siksek \cite{DNS2020} proved that elliptic curves over totally real cubic fields are ``modular" meaning that their $L$-functions match the $L$-function of the associated Hilbert modular forms. A key step to obtain this result is the study of points on a plane (singular) model of $X(\mathrm{b}5,\mathrm{ns}7)$. 

Banwait and Cremona \cite{BC2014} examined the failure of the local-to-global principle for the existence of $\ell$-isogenies between elliptic curves over number fields by, among others elements, determining a model for the exceptional modular curve $X_{S_4}(13)$. Zywina, in \cite{Z2020}, generalised the work of Banwait and Cremona, by relying on numerical approximation of pseudo--eigenvalues of Atkin--Lehner operators. Through his approach it is possible to determine $q$-expansions and models for modular curves. 

Box in \cite{B2021} described an algorithm \cite{B2021}*{Algorithm~4.13}, that has been implemented by the second named author, see \cite{ASSAFgit}, for computing the canonical model for $X_G/\Q$ in the case where $G$ has surjective determinant, $-I\in G$ and $G$ is normalised by $J:= \left(\begin{smallmatrix} -1&0\\0&1\end{smallmatrix} \right)$. In this algorithm, one first determines the $q$-expansions of a basis for the corresponding space of cusp forms and then a model, using the techniques developed by Galbraith \cite{G1996} when the genus is at least $ 2$. 
Box's algorithm presents the advantage that, for a finite groups $\mathcal{A}$ of the automorphism group of $X_G$, it is possible to determine a model for the quotient curve $X_G / \mathcal{A}$ directly, without computing $X_G$ first. Box's algorithm is another key ingredient to reach the conclusion of the main result of this article.

Notice that for degree larger than $3$ all smooth plane curves are non-hyperelliptic, see for example \cite{Hart}*{Ex. IV.5.1}. In \cite{BKX2013} the authors prove that for $N \geq 8$ all geometrically connected curve modular curves $X(N)$ defined over $\Q$ are neither hyperelliptic nor bielliptic.

Enge and Schertz \cite{ES2005} presented (singular) models for the modular curves $X_{0}(N)$ for $N$ the product of two arbitrary primes using Dedekind’s $\eta$ functions. 
Kodrnja in \cite{K2018}, relying on the embeddings in projective space through modular forms and modular functions presented by Muić in \cite{M2014} for computing models of modular curves, was able to find an explicit recipe to obtain plane (singular) models for all modular curves $X_{0}(N)$ for $N \geq2$. The equation of the model is the minimal polynomial of the modular function $\Delta(Nz)/\Delta$ over $\C(j)$, where $\Delta$ is the Ramanujan $\Delta$ function and $j$ is the modular $j$ function. Some plane (singular) models for modular curves $X_{0}(N)$ were already found by Hasegawa and Shimura in \cite{HS1999} using different ideas, in particular studying the gonality of modular curves. 

Borisov, Gunnells and Popescu \cite{BGP2001} showed that it is possible to determine explicitly an embedding of the modular curve $X_{1}(p)$ into $\P^{\frac{(p-3)}{2}}$, where $p\geq 5$ is a prime, using weight one Eisenstein series. The equation obtained is a (singular) quadratic equation. More recently, Baziz \cite{BA2010} proposed different (singular) models for $X_{1}(N)$ using $N$-division polynomials, and so with the advantage of keeping track explicitly of the corresponding pairs $(E,P)$ parametrised by the curve.

In this article we are interested only in modular curves as classically presented: projective complex algebraic curves corresponding to the compactification of the quotient space of the complex upper half plane by the action of a modular subgroup. Nevertheless, it is possible to define curves that are modular: a curve $C$ over $\Q$ is modular if it is dominated by $X_{1}(N)$ for some $N$. Moreover, if in addition the image of the jacobian of the curve in $J_{1}(N)$ is contained in the new subvariety of $J_{1}(N)$, then $C$ is new-modular. 
Under this definition, the modular curves associated to the classical modular groups $\Gamma_0(N)$ and $\Gamma_1(N)$, for some positive integer $N$, are curves that are modular. In particular there are infinitely many curves over $\Q$ that are modular and of genus $1$: elliptic curves over $\Q$ are modular. 
Baker, González-Jiménez, González and Poonen in \cite{BGGP2005} showed that for each genus $g \geq 2$, the set of curves over $\Q$ of genus $g$ that are new-modular curves is finite and computable. In particular, by analysing the automorphism group of the curve and the dominant map, they describe explicitly all curves that are new-modular of genus $2$, and construct a list of new-modular hyperelliptic curves of all genera (this list might be complete, but there are pathologies presented in the last sections of the aforementioned paper). 
 In \cite{GO2010} Gonz{\'a}lez-Jim{\'e}nez and Oyono gave an algorithm to compute explicit equations for non-hyperelliptic curves that are modular of genus $3$ over $\Q$. Moreover they conjectured that the list of  non-hyperelliptic curves that are new-modular and  of genus $3$ consists of  $44$ curves, and provided equations for all of them. The issue, as in \cite{BGGP2005}, is giving a bound for the coefficients of the modular forms involved.

\subsection{Structure of the paper}
In \S 2 we prove that there are a finite number of modular curves admitting a smooth plane model. To achieve this result, we bound the genus of such curves and notice that there is a finite number of congruence subgroups of any given genus. Moreover, we explicitly bound the level and the index of such groups. These results give us a finite list of groups corresponding to modular curves that may admit a smooth plane model. In \S 3 we discuss how to perform the computation of the canonical model of the relevant modular curves. In particular, we present the analysis regarding the runtime of the algorithm for computing $q$-expansions, with the precision required to prove the correctness of the resulting equations. 
Later, in \S 4 we present an algorithm that, given a canonical model of a non-hyperelliptic curve, checks whether the curve admits a smooth plane model and, if it is the case, computes it. Finally, in \S 5 we present our computations regarding Shimura type modular forms and modular curves.


\subsection{Acknowledgements} We thank the organisers of the conference "Arithmetic Aspects of Explicit Moduli Problems" held in Banff in June 2017 where this project was first mentioned. We thank Peter Bruin, Bas Edixhoven and Noam Elkies for insightful discussions during the aforementioned conference. We also thank John Voight for putting us in contact with each other. Last but not least, we thank Wouter Castryck and all the referees that read our paper and helped us to improve its results and exposition; in particular, to the one detecting a crucial error in the first version of Section \ref{sec:4}. 

The research of the first and third author is partially funded by the Melodia ANR-20-CE40-0013 project.
The second author was supported by a Simons Collaboration grant (550029, to John Voight).

\section{A bound for the genus}

In this section we prove the first two parts of the main theorem, Theorem~\ref{thm:main}. 

\begin{theorem}
\label{finite}
There are a finite number of modular curves admitting a smooth plane model. Moreover, the degree of such model is less or equal to $18$.
\end{theorem}

\begin{proof} The genus–degree formula tells us that a smooth plane curve of degree $d$ has genus $g=\frac{(d-1)(d-2)}{2}$. The gonality (over the algebraic closure) of a smooth plane curve of degree $d$ is $d-1$, see \cite{COKA90}*{Theorem A}. The gonality  of a modular curve of genus $g$ is greater or equal to $\frac{1}{2}\cdot 975(g-1)/4096 $, see  \cite{POONEN07}*{Remark 1.2} and \cite{BGGP2005}*{Remark 4.5}.  Therefore, for a modular curve admitting a smooth plane model we have that $975d^{2} - 19309 d + 16384\leq 0$ and so 
$$1\leq d\leq 18 \quad \mbox{ and } \quad g\in \{0, 1, 3, 6, 10, 15, 21, 28, 36, 45, 55, 66, 78, 91, 105, 120, 136\}.$$
There are a finite number of modular curves of a given genus, see \cite{COPA84}, so there are a finite number of modular curves admitting a smooth plane model.
\end{proof} 

For degree $1$ and $2$, i.e. genus $0$, 
the list of 
levels is given in \cite{COPA84}*{Table 4.24}.

For degree $3$, i.e. genus $1$, the complete list of the relevant congruence subgroups is given in \cite{CUPA03}.

For degree $4$ we need to consider curves of genus $3$. The non-hyperelliptic ones are given by smooth plane quartics. Indeed, we find modular curves of genus 3 admitting a smooth plane model of degree $4$, see Table~\ref{tab: plane quartic genus 3} for the complete list we have computed. 

Nevertheless, the following question arises naturally:

\begin{question}
Is there any modular curve of genus greater than $3$ admitting a smooth plane model? 
\end{question}

For degrees $5$ and $6$ we did not find any example of a modular curve admitting a smooth plane model, restricting to Shimura type modular curves, see \S 5.

For each genus up to $24$ the complete explicit list of congruence subgroups of $\PSLZ$ is known: see \cite{CUPA03}*{Theorem 2.8} and the associated website\footnote{\href{https://mathstats.uncg.edu/sites/pauli/congruence/}{https://mathstats.uncg.edu/sites/pauli/congruence/}}.
\vspace{2ex}

One way to count how many modular curves may admit a smooth plane model is to count congruence subgroups of $\PSLZ$ whose index is bounded in terms of the degree of the model, as follows.

\begin{proposition}
\label{finiteindex2}
The index $\iota$ of a congruence subgroup in $\PSLZ$ whose associated modular curve admits a smooth plane model of degree $d \geq 3$ satisfies
\begin{equation}\label{bd}
6(d-1)(d-2) -12 \leq\iota \leq 101(d-1)
\end{equation} 
\end{proposition}
\begin{proof} On the one hand, combining \cite{Z84}*{Theorem 3} (see also \cite{HS1999}*{Theorem 4.3}) and an improvement presented in \cite{CUPA03} due to Kim and Sarnak \cite{Kim}*{Appendix 2}, the index of a congruence subgroup in $\PSLZ$ is bounded by $101$ times the gonality of the corresponding modular curve. By assumption the modular curves admits a smooth plane model, so its gonality is $d-1$. The index is therefore bounded by $101 (d-1)$.

On the other hand, the genus $g$ of a modular curve admitting a smooth plane model of degree $d \geq 3$ satisfies $g=\frac{(d-1)(d-2)}{2}>0$ and $g\leq 1 + \frac{\iota}{12}$, where $\iota$ is the index of the corresponding congruence subgroup, see \cite{DISH}*{Theorem 3.1.1}.

Therefore the index $\iota$ is bounded above and below as in Equation \ref{bd}.
\end{proof} 

\begin{remark}
The coefficient $101$ used in Proposition \ref{finiteindex2} is obtained by taking the floor of a rational number $\alpha = 2^{15} / 325$. A sharper upper bound can be obtained by rounding only after multiplication.
\end{remark}

\begin{remark}
The result of Proposition \ref{finiteindex2} together with the bound for the degree presented in Theorem \ref{finite}, and the previous remark, implies in each case the following lower and upper bounds for the index $\iota$:
\begin{table}[h]
\centering
\begin{tabular}{|c|c|c||c|c|c|}
\hline
degree & genus & index bound                  & degree & genus & index bound                  \\ \hline
3      & 1     & $0\leq \iota \leq201$   & 11      & 45    & $528\leq\iota \leq1008$ \\
4      & 3     & $24\leq\iota \leq302$  & 12      & 55    & $648\leq\iota \leq1109$ \\
5      & 6     & $60\leq\iota \leq403$  & 13     & 66    & $780\leq\iota \leq1209$ \\
6      & 10    & $108\leq\iota \leq504$ & 14     & 78    & $924\leq\iota \leq1310$\\
7      & 15    & $168\leq\iota \leq604$ & 15     & 91    & $1080\leq\iota \leq1411$\\ 
8      & 21    & $240\leq\iota \leq705$ & 16     & 105    & $1248\leq\iota \leq1512$\\
9      & 28    & $324\leq\iota \leq806$ & 17     & 120    & $1428\leq\iota \leq1613$\\
10     & 36    & $420\leq\iota \leq907$ & 18    & 136    & $1620\leq\iota \leq1714$\\ \hline
\end{tabular}
\caption{Index bounds}
\label{tab: level index}
\end{table}
\end{remark}

The logarithm of the number of congruence subgroups in $\PSLZ$ of index bounded by $1714$ is approximately $1897$, see \cite{BNP21}*{Proposition 8.1}. Therefore naively listing all subgroups would be not feasible, and the list of \cite{CUPA03} contains only groups of genus less than or equal to $24$.


%
%
%
%
%

Let us also remark that for any given genus we can bound the level $N$ of the congruence subgroups occurring using the following formula, due to Cox and Parry~\cite{COPA84}*{Equation (4.22)}, 
\begin{equation} \label{eq:level bound}
N\leq
\begin{cases}
168 & \text{if }g=0\\
12g + \frac{1}{2}(13\sqrt{48g+121}+145)& \text{if }g\geq 1
\end{cases}
\end{equation}
Analysing the genera in Theorem \ref{finite} we produce the level bounds appearing in Table~\ref{tab: level bounds}.

\begin{table}[ht]
\begin{tabular}{|c|c||c|c|}
\hline
genus & level bound & genus & level bound \\ \hline
1     & 169         & 45    & 922         \\
3     & 214         & 55    & 1074         \\
6     & 275        & 66    & 1237        \\
10    & 351         & 78    & 1412     \\   
15    & 441         & 91    & 1600  \\
21    & 542         & 105    & 1799  \\
28    & 657         & 120    & 2010  \\
36    & 784         & 136   & 2234  \\\hline   

\end{tabular}
\caption{Level bounds}
\label{tab: level bounds}
\end{table}

It remains to check this finite number of possibilities, a task which we proceed to describe in the rest of the paper.

\section{Computing Modular Curves}

Let $\Gamma \subseteq \PSLZ$ be a congruence subgroup of level $N$. Then the modular curve $X_{\Gamma}$ can be given the structure of an algebraic curve over $\Q(\zeta_N)$. This structure depends on the choice of a group $G \subseteq \GL_2(\Z / N \Z)$ such that the projection of its pullback to $\SL_2(\Z)$, denoted by $PG$, coincides with $\Gamma$. We denote such a model by $X_G$. 
The Galois action on the connected components of the curve $X_G$ is given by the homomorphism $\sigma_d \mapsto \left( \begin{smallmatrix}d & 0 \\ 0 & 1 \end{smallmatrix} \right) $, where 
$\sigma_d(\zeta_N) = \zeta_N^d$. Therefore, the field of definition of $X_G$ is the fixed field of $\det(G) \subseteq (\Z / N \Z)^{\times} = \Gal(\Q(\zeta_N) | \Q)$, where $\det$ denotes the usual determinant map from $\GL_2(\Z / N \Z)$ to $(\Z / N \Z)^{\times}$. The connected components of the curve $X_G$ are indexed by $(\Z / N \Z)^{\times} / \det(G)$, and each component is defined over the field $\Q(\zeta_N)^{\det(G)}$.
In particular, $X_G$ is geometrically connected and defined over $\Q$ if and only if $\det(G) = (\Z / N \Z)^{\times}$. Therefore, $X_{\Gamma}$, which is one of the components of $X_G$, admits a model over $\Q$ only if there exists $G \subseteq \GL_2(\Z / N \Z)$ such that $PG = \Gamma$ and 
$\det(G) = (\Z / N \Z)^{\times}$.

The methods of Galbraith and Box, described briefly in the introduction, for computing modular curves use duality with modular symbols, and therefore require $G$ also to be of real type, i.e. such that $JGJ = G$, where $J = \left( \begin{smallmatrix}-1&0\\0&1 \end{smallmatrix} \right)$. Since $J$ acts via complex conjugation on the Fourier coefficients of modular forms, it is equivalent to requiring the Fourier coefficients to be fixed by complex conjugation. 

We therefore restrict our attention to congruence subgroups $\Gamma$ such that there exists $G$ of real type with surjective determinant and $PG = \Gamma$. 
Note further that for these groups, when the degree is prime to $3$, it suffices to check one such model $X_G$ by \cite{BBLG2019}*{Corollary 2.7}.
In the range of degrees we are interested in, the only relevant case is that of degree $6$, i.e. genus $10$. In this case, for groups of Shimura type, the curve always admits a rational point, and so it is again enough to consider a single model by \cite{BBLG2019}*{Corollary 2.2}. For the other congruence subgroups of genus $10$ for which we compute the curve, we verify that the resulting curves indeed have rational points, hence in these cases it also suffices to check a single model.

Our method of enumerating these subgroups of specific genus is to run over the finite list of conjugacy classes of congruence subgroups of this genus in $\PSLZ$, and for a representative $\Gamma \subseteq \PSLZ$, we look at the projection of its pullback $H \subseteq \SL_2(\Z / N \Z)$. As for any compatible $G \subseteq \GL_2(\Z / N \Z)$, $H$ will be a normal subgroup, we start by looking for a conjugate $H'$ of $H$ in $\GL_2(\Z / N \Z)$ which satisfies $JH'J = H'$, or equivalently $J \in N(H)$, where the normalization takes place in $\GL_2(\Z / N \Z)$. Since $N(gHg^{-1}) = gN(H)g^{-1}$, it suffices to consider conjugates of $N(H)$, and look for one which contains $J$.
We then note that if $G$ is such that $G \cap SL_2(\Z / N \Z) = H$, then $H \trianglelefteq G$, so that $G \subseteq N(H)$. Thus, looking for $G$ with surjective determinant amounts to enumerating the subgroups of $N(H) / H$ of order $\phi(N)$. 

In Table~\ref{tab: congruence subgroup count} we list how many congruence subgroups $\Gamma$ exist, up to conjugacy, for each degree $3 \le d \le 8$, and how many of these admit a model $G \subseteq \GL_2(\Z / N \Z)$ of real type with surjective determinant.
In Table~\ref{tab: congruence subgroup count} we also record the number of groups of Shimura type of each degree $3 \le d \le 8$.

\begin{table}[ht]
\begin{tabular}{|c|c|c|c|c|}\hline
    degree        &      genus       &     \begin{tabular}{c} congruence \\ subgroups \end{tabular}        &     \begin{tabular}{c} real type \& \\ surjective det \end{tabular}       &    Shimura type         \\ \hline
    3      &      1      &    163        &        108     & 38            \\
      4      &      3       &     241        &        160     &          26   \\
       5     &      6       &     175        &      74       &         8\\
	6 	& 10 		 &  235	& 	120		&            17 \\
	7 & 	15 		&  485	& 	244		& 23 \\
	8 &     21      & 729   &   431     & 55 \\
	\hline
\end{tabular}
\caption{Congruence subgroups of low genus}
\label{tab: congruence subgroup count}
\end{table}

The methods we use for computing equations of modular curves make use of explicit computation of the $q$-expansions and the canonical map. We briefly recall the map and its properties. 

\subsection{The canonical map} Let $k$ be a perfect field. Let $C/k$ be a smooth projective curve of genus $g\geq2$ with canonical divisor $K$. Let $\{z_0,...,z_{g-1}\}$ be a basis defined over $k$ of the Riemann-Roch space $\mathcal{L}(K)$. The canonical map of $C$ is given by
$$
\phi_K:\,C\rightarrow\mathbb{P}^{g-1}, \quad\,P\mapsto(z_0(P):...:z_{g-1}(P)).
$$
The curve $C$ is non-hyperelliptic if and only if $\phi_K$ is an embedding. In this case $\phi_K(C)$ is defined over $k$ and it is unique up to a linear transformation of $\mathbb{P}^{g-1}$. Otherwise, when $\phi_K$ is not an embedding, the curve $C$ is hyperelliptic and $\phi_K$ is the quotient by the hyperelliptic involution: $\phi_K(C)\simeq\mathbb{P}^1$.


\begin{theorem}\label{thm:NEP}(Noether-Enriques-Petri, \cite{Noether})
Let $C$ be a smooth projective non-hyperelliptic curve of genus $g$. The homogeneous ideal defining the canonical curve $\phi_K(C)\subseteq\mathbb{P}^{g-1}$ is generated by its elements of degree 2, except in the following cases:
\begin{itemize}
    \item $g=3$, so $C$ is a smooth plane quartic.
    \item $g\geq 4$ and $C$ is a trigonal curve. In this case an element of degree 3 is also needed to generate the ideal.
    \item $g=6$ and $C$ is a smooth plane quintic. Again in this case an element of degree 3 is also needed.
\end{itemize}
\end{theorem}

Therefore, to compute an equation for the modular curve, using the identification $S_2(\Gamma, \Q(\zeta_N))^G \simeq \Omega^1(X_G)$, it suffices to compute $q$-expansions up to sufficient precision and look for relations in low degrees. We proceed by describing first the required precision. 

\subsection{Bounds}




In order to distinguish modular forms we will use a finite number of coefficients of the associated $q$-expansions thanks to the following result due to Sturm \cite{STU87}*{Theorem~1}, see also \cite{STE07}*{Section~9.4}. Let us recall that for a congruence subgroup $\Gamma\subseteq \SL_2(\Z)$ the width of the cusp $\infty$ is the positive integer $h$ defined by $\left(\begin{smallmatrix}1&h \Z\\0&1\end{smallmatrix}\right)=\Gamma\cap\left(\begin{smallmatrix}1&\Z\\0&1\end{smallmatrix}\right)$.

\begin{theorem}[{\cite{STU87}*{Theorem~1}}]
Let $\Gamma$ be a congruence subgroup of $\SL_2(\Z)$. Let $h$ be the width of the cusp $\infty$ for $\Gamma$. Let $f$ be a modular form on $\Gamma$ of weight $\wt$, with coefficients in a discrete valuation ring $R$ contained in $\C$. Let $\F$ be the residue field of $R$. Suppose that the image $\sum a_n q^{n/h}$ in $\F[[q^{1/h}]]$ of the $q$-expansion of $f$ has $a_n=0$ for all $n\leq \wt [\SL_2(\Z):\Gamma]/12$. Then $a_n=0$ for all $n$, i.e.\ $f$ is congruent to $0$ modulo the maximal ideal of $R$.
\end{theorem}


Moreover, we can state the following corollary, derived from an observation at the end of \cite{STU87} and stated in this form in \cite{RA09}:
\begin{corollary}[\cite{RA09}*{Theorem~2.1}] \label{cor:Sturm bound}
Under the same hypotheses of the theorem above, let us assume furthermore that $f$ is a cusp form. If the image $\sum a_n q^{n/h}$ in $\F[[q^{1/h}]]$ of the $q$-expansion of $f$ has $a_n=0$ for all $n\leq \wt [\SL_2(\Z):\Gamma]/12-\#(\text{cusps})$. Then $a_n=0$ for all $n$, i.e.\ $f$ is congruent to $0$ modulo the maximal ideal of $R$.
\end{corollary}
The integer $\wt [\SL_2(\Z):\Gamma]/12$ (resp. $\wt [\SL_2(\Z):\Gamma]/12-\#(\text{cusps})$) is known as the Sturm bound (resp. Sturm bound for cusp forms) and we will use the notation $B(\Gamma,\wt)$ (resp. $B(\Gamma,\wt)_c$) to refer to such a bound. 

\subsection{Groups of Shimura type}
For groups of Shimura type, the methods described in \cite{A2022} can be used to compute the $q$-expansions. Alternatively, conjugating by 
$ \alpha_t = \left( \begin{smallmatrix} 1 & 0 \\ 0 & t \end{smallmatrix} \right) $, we see that 
$$
\Gamma_1(Nt) \subseteq \alpha_t \Gamma(H,t) \alpha_t^{-1} \subseteq \Gamma_0(Nt).
$$
Moreover, if we decompose the space by Dirichlet characters as
$$
S_2(\Gamma_1(Nt)) = \bigoplus_{\chi : (\Z / Nt \Z)^{\times} \to \C^{\times} } S_2(\Gamma_0(Nt), \chi),
$$ 
then we obtain
$$
S_2(\alpha_t \Gamma(H,t) \alpha_t^{-1}) =  \bigoplus_{\chi : \chi(H) = 1 } S_2(\Gamma_0(Nt), \chi).
$$
The $q$-expansions for modular forms in the spaces in this decomposition are then straightforward to compute. 

In order to compute equations for all modular curves of Shimura type of genus $1,3,6,10,15$, we will need to compute weight $2$ cusp forms and then check quadratic and cubic relations, according to Theorem 
\ref{thm:NEP}. The number of coefficients of the $q$-expansions of the weight $2$ cusp forms needed to certify the computation performed, is equal to $B(\Gamma,\wt)_c$, where $\wt$ is either $4$ or $6$. 

\begin{proposition}
The level of a congruence subgroup in $\PSLZ$ associated to a Shimura type modular curve admitting a smooth plane model is bounded by $1709$. More precisely, we can bound the level for each genus as shown in Table~\ref{tab: level shimura}.
\end{proposition}
\begin{proof} 
A congruence subgroup $\Gamma$ in $\PSLZ$ corresponding to a Shimura type modular curve  is contained in $\Gamma_{0}(M)$ and contains $\Gamma_{1}(M)$ for an appropriate level $M$, after conjugation in $\GL_2(\hat{\Z})$. Its index is bounded by $1714$ as in Table~\ref{tab: level index} and so by direct computation the level is bounded by $1709$.
\end{proof} 

\begin{table}[ht]
\centering
\begin{tabular}{|c|c||c|c|}
\hline
genus & level bound & genus & level bound                  \\ \hline
1     & 199    & 45      & 997 \\
3     & 293    & 55      & 1103 \\
6     & 401    & 66      & 1201 \\
10    & 503    & 78      & 1307 \\
15    & 601    & 91      & 1409 \\
21    & 701    & 105     & 1511\\
28    & 797    & 120     & 1609 \\
36    & 887    & 136     & 1709 \\\hline
\end{tabular}
\caption{Level bounds for Shimura type modular curve admitting a smooth plane model}
\label{tab: level shimura}
\end{table}

\subsection{Other congruence subgroups}
For groups that are not of Shimura type,  we apply the (generalization of the) method of twists described by Box in \cite{B2021}. We note that Box uses an auxiliary divisor $M$ of the level $N$ such that $G_M = B_0(M)$, where $B_0(M)$ is the Borel subgroup of $\GL_2(\Z / M \Z)$, but this constraint can be relaxed to allow for $B_1(M) \subseteq G_M \subseteq B_0(M)$, where $B_1(M)$ is the unipotent subgroup of $B_0(M)$, by decomposing according to the action of Dirichlet characters. More precisely, if $G' \subseteq \GL_2(\Z / N \Z)$ is such that $G \trianglelefteq G'$ and $G' / G$ is abelian, we can decompose according to the characters of $G' / G$, namely $$ S_2(G) = \bigoplus_{\eps : G'/G \to \C^{\times}} S_2(G', \eps). $$ 
In the cases where $G' = B_0(M)$ and $\Gamma_1(M) \cap \Gamma(K) \subseteq G$ for some relatively prime $K,M$, such that $KM=N$,  we may further identify 
$$ 
S_2(G',\eps) = 
\bigoplus_{\substack{\chi : \left( \Z / MK^2 \Z \right)^{\times} \to \C^{\times} \\ \chi \vert_{(K \Z + 1)/ MK^2 \Z} = \eps}} 
S_2(\Gamma_0(MK^2), \chi)^{G'=\epsilon}
$$
by conjugating with $\alpha_K = \left(\begin{smallmatrix} 1 & 0 \\ 0 & K \end{smallmatrix} \right)$.
We can then create the space of modular symbols corresponding to $S_2(\Gamma_0(MK^2), \chi)$, and cut out the subspace on which $G'$ acts via $\epsilon$ by a method similar to the algorithm described in \cite{B2021}*{Algorithm~4.11}. Putting together all these elements, we obtain an algorithm that, given a group $G$ such that $B_1(M) \subseteq G_M \subseteq B_0(M)$, returns the $q$-expansions of a basis for the space of cusp forms $S_2(G)$.
Denote by $\pi_M : \GL_2(\Z / N \Z) \to \GL_2(\Z / M \Z)$ the natural projection map. 

\begin{proposition}
The running time complexity of the algorithm described above for a group $G = \pi_M^{-1}(G_M) \cap \pi_K^{-1}(G_K)$ with $B_1(M) \subseteq G_M \subseteq B_0(M)$ of genus $g$ is given by
\begin{equation*}
    \tilde{O}\left([G_M : B_1(M)] (M^3K^6 + MK^4g^2) \right).
\end{equation*}
For a group of Shimura type $G= G(H,t)$, it is given by $\tilde{O}\left(\frac{\phi(N)}{|H|} Nt^2g^2 \right)$.
\end{proposition}

\begin{proof}
We note that the complexity of the algorithm is dominated by the linear algebra operations performed in these spaces of modular forms. Specifically, since the algorithm requires computing the Hecke decomposition, to obtain the modular symbols corresponding to the eigenform, our complexity is dominated by~$\tilde{O}(d^3 + dL^2)$, where $d = \dim S_2(\Gamma_0(MK^2), \chi)$ is the dimension of the space and $L$ is the precision required for the $q$-expansions, see \cite{BBCCCDLLRSV2020}*{Table 5.2.3}.
By Corollary~\ref{cor:Sturm bound}, as the cusp width $h$ is bounded by $K$, we see that the required precision to ascertain our linear relations indeed hold is bounded by $L \le K B(\Gamma, \wt)_c$, where $\wt$ is the maximal weight in which we look for an equation and $\Gamma$ is the pullback of $G$ to $\SL_2(\Z)$.
By Theorem~\ref{thm:NEP}, $\wt \in \{4, 6\}$ in all considered cases, with $\wt = 6$ occurring only if $g \le 6$. 
Finally, by \cite{Z1991}*{Theorem 2.3}, the index 
$[\SL_2(\Z) : \Gamma] = O(g)$, hence $L = O(Kg)$. 
As $d = O(MK^2)$ (see \cite{M2005} for a more precise and detailed asymptotic analysis), it follows that the running time complexity of the algorithm on the space corresponding to each Dirichlet character $\chi$ is $\tilde{O}(M^3K^6 + MK^4g^2)$, and summing over all Dirichlet characters we obtain the result. For a group of Shimura type, we can simply compute for each of the direct summands the Hecke operators up to the required precision, which now satisfies $L = O(tg)$, as the cusp width at $\infty$ is precisely $h = t$.
\end{proof}


As a result, when looking for smooth plane models of general congruence subgroups, we will have to restrict ourselves to reasonable ranges of the parameter $MK^2$. We therefore treat in this paper only groups that are of Shimura type or such that $MK^2 \le 500$. We further note that different representatives in the conjugacy class of $\Gamma$, and different groups $G$ which pull back to $\Gamma$ give rise to different values of $M,K$. We find for each conjugacy class of the congruence subgroup $\Gamma$, a corresponding group $G$ with the maximal value of $M$ (and so the minimal for $MK^2$). Moreover, by their definition, for groups of Shimura type we may choose $M = N/t $ and $K = t$, making them the easiest to compute using this method as well.

\section{From a canonical model to a smooth plane model}\label{sec:4}
In this section we propose an algorithm to check whether a smooth irreducible projective curve $C$ of genus $g$ and defined over a perfect field $k$ does not admits a smooth plane model over $\bar{k}$. In the case where we cannot rule out this possibility, we propose a strategy to compute such smooth plane model. We do not focus on the minimal fields of definitions for these models, but we point the interested reader to \cite{BBLG2019}.

\subsection{The low genus cases}
For $g=0,1$ there is always a smooth plane model. For genus $2$, or more generally, for hyperelliptic curves, there is never one. For genus $3$ non-hyperelliptic there is always a smooth plane model and it is given by the canonical model. The next genus to check is $6$, for which we have another necessary condition in order to have a smooth plane model of degree $5$: the canonical ideal $I_C$ defining $\phi_K(C)$ is not generated only by degree $2$ elements, see Theorem \ref{thm:NEP}. Still in this situation we need to distinguish between trigonal curves and smooth plane quintics, see Subsection \ref{subsec:g6} for a detailed example. Let us recall the following classical result, coming from the description of the regular differential forms of a smooth plane curve: The canonical model of a degree $d$ smooth plane curve is given by the composition of $C\hookrightarrow\mathbb{P}^2$ with the $(d-3)$-Veronese embedding $\mathbb{P}^2\hookrightarrow\mathbb{P}^{g-1}$.

\begin{lemma}\label{lem:g6}
Let $C$ be a smooth plane quintic curve. Then the degree $2$ elements of the canonical ideal $I_C$ defining $\phi_K(C)$ define a $\mathbb{P}^2$. A bijective parametrization of it,  evaluated at a degree $3$ non-trivial generator of $I_C$, gives the smooth plane quintic model. 
\end{lemma}
\begin{proof}
In the quintic case $C:\,F(x,y,z)=0\subseteq\mathbb{P}^2$ with $\operatorname{deg}(F)=5$ and the Canonical model is given by the composition with the $2$-Veronese embedding $\mathbb{P}^2\hookrightarrow\mathbb{P}^5$. The canonical image $\phi_K(C)$ is generated by the equations defining $\phi_K(\mathbb{P}^2)$ that can all be taken of degree $2$ and the $3$ degree $3$ equations, corresponding to $xF(x,y,z)=0$, $yF(x,y,z)=0$ and $zF(x,y,z)=0$.
\end{proof}

We deal next with the higher genus situations.

\subsection{The minimal free resolution}

A smooth curve $C$  of genus $g$ admits a smooth plane model if and only if it has a (unique up to linear equivalence) very ample complete $g^2_d$-linear series, i.e. a very ample divisor $D$ such that $\text{deg}(D)=d$ and $\ell(D)=3$. Given a basis $\{x,y,z\}$ of $\mathcal{L}(D)$, the plane model is given by the image of $C\rightarrow\mathbb{P}^2:\,P\mapsto(x(P):y(P):z(P))$. 

\begin{theorem}[\cite{Green}*{Appendix}]\label{thm:Koszul}
If a smooth curve $C$ of genus $g=\frac{(d-1)(d-2)}{2}$ with $d\geq5$ and canonical divisor $K$, has a $g^2_d$-linear series then the Koszul cohomology group $\mathcal{K}_{\frac{(d-3)(d-2)}{2},1}(C,K)\neq 0$.
\end{theorem}

This theorem proves a special case of one of the directions of Conjecture 5.1 in \cite{Green}. In terms of graded Betti numbers \cite{schreyer}*{p. 84} we have: $$\operatorname{dim}(\mathcal{K}_{\frac{(d-3)(d-2)}{2},1}(C,K_C))=\beta_{d-4,d-2}.$$

Let $C/\mathbb{C}$ be a smooth curve of genus $g$ and $\phi_K(C)$ its image by the canonical map given by the ideal $I_C$ in $S=\mathbb{C}[z_0,z_1,...,z_{g-1}]$. Let $S_C=S/I_C$ be the homogeneous coordinate ring of $\phi_K(C)$. We consider the minimal free resolution:
$$
0\leftarrow S_C\leftarrow S\leftarrow F_1\leftarrow F_2 \leftarrow ... \leftarrow F_{g-2}\leftarrow 0.
$$

Noether proved that $F_i=S(-i-1)^{\beta_{i,i+1}}\oplus S(-i-2)^{\beta_{i,i+2}}$ for $i=1,...,g-3$, i.e. that $F_i$ is a module generated by elements of degree $i+1$ and $i+2$. These Betti numbers can be computed with \verb+Magma+ \cite{Magma}. In order to speed up these calculations, we compute the Betti numbers for the reduction of the curve modulo a prime of good reduction: in this case the Betti numbers are the same for both curves, see \cite{milneLEC}*{Thm. 20.5}.  

When $\beta_{d-4,d-2}\neq0$, we still need to check whether a smooth plane model exists. As in the proof of Lemma \ref{lem:g6}, when $g\geq6$ the ideal $I_C$ is generated by the degree 2 equations defining $\mathbb{P}^2\hookrightarrow\mathbb{P}^{g-1}$ by the $(d-3)$-Veronese embedding plus the (this time) degree 2 equations $x^{a}y^bz^cF(x,y,z)=0$ with $a+b+c=d-6$. In order to recover the putative smooth model we aim to determine the degree 2 equations defining the $\mathbb{P}^2$. Then we compute a bijective parametrization and plug it into any other equation of $I_C$, not defining the $\mathbb{P}^2$, and, therefore, we should obtain the smooth plane model we are looking for. If not such a model is found, it means that it does not exists. This strategy to recover the $\mathbb{P}^2$ is the one in the proof of Theorem 4.1 in \cite{schreyer} that gives a proof of the reverse implication of Conjecture 5.1 in \cite{Green} for $d=6$. The idea is to recover the exceptional surface, so the $\mathbb{P}^2$, by finding relations with a certain shape and the standard basis techniques presented in the Appendix of \cite{schreyer}. We present an implementation of this algorithm in \cite{ASSAFgit}.

\subsection{The algorithm}
Following the discussion in the previous subsections, we present an algorithm,  Algorithm~\ref{alg:cantoplane}, which allows to determine whether a curve admits a smooth plane model.

\begin{algorithm}[hbt!]
\caption{Existence of a smooth plane model}\label{alg:cantoplane}
\KwData{$C/k$ a genus $g$ curve given with its canonical model}
\KwResult{determining whether $C$ admits a smooth plane model and, when possible, returning such model}
$M \gets 0 $\;
\eIf{$g$ is $0$ or $1$}
{$T \gets true $\;}
{\eIf{$g$ is $3$}
{\eIf{$g(\phi_K(C))=3$}{$T \gets true$\;$M \gets \phi_K(C)$\;}{$T \gets false$\;}}{
\eIf{$g$ is $6$}{\eIf{$I_C$ generated by quadrics}{$T \gets false$\;}{compute $M$ with Lemma \ref{lem:g6}\; 
\eIf{$M$ is a smooth plane quintic}{$T \gets true $\;}{$T \gets false $\;}}}{
$T \gets false$\;
\If{$\exists d\in \N$ with $g=(d-1)(d-2)/2$ and $g(\phi_K(C))\neq 0$ }{Apply Theorem~\ref{thm:Koszul}\;
\If{$\beta_{d-4,d-2}\neq0$ }{compute $M$ with Schreyer's strategy \cite{schreyer}\; \If{$M$ is smooth}{$T \gets true$\;}}
}
}
}
}
\Return{$T,M$}
\end{algorithm}

\subsection{Other strategies}
Sometimes, in order to prove that a certain curve does not admit a smooth plane model, we can try some less computationally expensive techniques. For instance, when the curves under considerations have some involutions:

\begin{theorem}\label{thm:Harui}[Remark 2.1 (i) \& Theorem 2.2 with $n=2$ in \cite{Harui}] Let $C$ be a smooth plane curve of degree $d$ and $\sigma$ an involution of $C$. Then the involution $\sigma$ has $f=d+\frac{1-(-1)^d}{2}$ fixed points and the quotient $C/\langle\sigma\rangle$ has gonality $\lfloor\frac{d}{2}\rfloor$.
\end{theorem}

Other ways of using the knowledge of some quotients to prove the non-existence of smooth plane models are the following results:

\begin{lemma}\label{lem:hyper}
Let $C$ be a smooth curve admitting  a degree $n$ morphism to a hyperelliptic curve. Then $C$ does not admit a smooth plane model of degree greater than $2n+1$. 
\end{lemma}
\begin{proof}
The gonality of a smooth plane curve of degree $d$ is $d-1$ and the gonality of a hyperelliptic curve is $2$.
\end{proof}

\begin{theorem}\label{thm:Greco}[Theorem 3.1 with $r=1$ in \cite{Greco}] A smooth plane curve of degree $d$ does not admit any rational map to $\mathbb{P}^1$ of degree $n$ such that $$(a - 1)d + 1 \le n \le ad - (a^2 + 1)$$ for some $a\in\mathbb{N}$.

\end{theorem}

\section{Computations}
We compute equations for all modular curves of Shimura type of genus $1,3,6,10$ and $15$, i.e. possibly admitting a smooth plane model of degree $3,4,5,6$ or $7$. We run the algorithm in previous section on all of them, in order to check which ones do admit a smooth plane model. 

We also compute equations for modular curves of these genera which are not of Shimura type, when the congruence subgroup is $\Gamma = PG$, with $G \subseteq \GL_2(\Z / N \Z)$ satisfies $G = \pi_M^{-1}(G_M) \cap \pi_K^{-1}(G_K)$ with $B_1(M) \subseteq G_M \subseteq B_0(M)$ and $G_K \subseteq GL_2(\Z / K \Z)$, and such that $MK^2 \le 500$.
Note that $(MK^2)^3$ is the dominant factor in the running time complexity, and indeed for larger values of $MK^2$, the linear algebra becomes the bottleneck.

We discuss the results in the following subsections. We use the congruence subgroup labels introduced in \cite{CUPA03}. All computations were done using \verb+Magma+ \cite{Magma} and the full results are available online at \cite{ASSAFgit}.

\subsection{Genus 1}
In the case of modular curves of genus $1$ they always admit a smooth plane model of degree $3$. However, it is not given by the canonical model since they are not non-hyperelliptic. In this case we note that all groups of Shimura type give rise to elliptic curves, since the cusp at $\infty$ is rational. We may further compute models for some of the other congruence subgroups, using the methods in \cite{RZB2015} to obtain the $j$-map and a model for the curve. 
For example, we see that for the congruence subgroups of level $6$, the groups labeled 6A1, 6C1, 6D1 all yield elliptic curves, with equations $y^2 = x^3-27$, $y^2 = x^3 + 1$ and $y^2 = x^3 + 1$. 
We also find the $q$-expansion of the unique eigenform for all $98$ congruence subgroups of genus $1$ for which $MK^2 \le 500$. 

\subsection{Genus 3}
Among the $26$ groups of Shimura type of genus $3$, we find that there are $11$ modular curves which are hyperelliptic, these are listed below. We note that $7$ of these curves belong to the $X_0(N)$ family, and indeed we recover the models of Galbraith \cite{G1996} for all these curves except for $X_0(35)$ and $X_0(41)$ that we find different ones. These, of course, give isomorphic curves to the Galbraith ones.

\renewcommand{\arraystretch}{1.33}

\begin{table}[ht]
\begin{tabular}{|c|c|c|}\hline
    \bf{label}  &  \bf{name}    &    \bf{Hyperelliptic model }          \\ \hline
    12K3        &               &   \small $ y^2 = x^8 + 14x^4 + 1$          \\
    20J3        &               &   \small $ y^2 = x^8 + 8x^6 - 2x^4 + 8x^2 + 1$   \\
    21D3        &               &   \small $ y^2 = x^8 - 6x^6 + 4x^5 + 11x^4 - 24x^3 + 22x^2 - 8x + 1$ \\
	24V3 	    &               &   \small $ y^2 = x^8 + 14x^4 + 1$ \\
	30K3        &   $X_0(30)$   &   \small $ y^2 = x^8 + 6x^7 + 9x^6 + 6x^5 - 4x^4 - 6x^3 + 9x^2 - 6x + 1 $ \\  
	33C3        &   $X_0(33)$   &   \small $ y^2 = x^8 + 10x^6 - 8x^5 + 47x^4 - 40x^3 + 82x^2 - 44x + 33$ \\
	35A3        &   $X_0(35)$   &   \small $ y^2 = x^8 - 12x^7 + 50x^6 - 108x^5 + 131x^4 - 76x^3 - 10x^2 + 44x - 19 $ \\
	39A3        &   $X_0(39)$   &   \small $ y^2 = x^8 - 6x^7 + 3x^6 + 12x^5 - 23x^4 + 12x^3 + 3x^2 - 6x + 1$ \\
	40F3        &   $X_0(40)$   &   \small $ y^2 = x^8 + 8x^6 - 2x^4 + 8x^2 + 1$ \\
	41A3        &   $X_0(41)$   &   \small $ y^2 = x^8 - 12x^7 + 48x^6 - 82x^5 + 60x^4 - 8x^3 - 27x^2 + 16x - 4$ \\
	48J3        &   $X_0(48)$   &   \small $ y^2 = x^8 + 14x^4 + 1$ \\
	\hline
\end{tabular}
\caption{Hyperelliptic Shimura type modular curves of genus $3$}
\label{tab: hyperelliptic genus 3}
\end{table}

\begin{remark}The curves corresponding to the groups labeled 12K3, 24V3 and 48J3 are isomorphic. The curves corresponding to the groups labeled 20J3 and 40F3 are isomorphic. No other curves in Table \ref{tab: hyperelliptic genus 3} are isomorphic. 
This phenomenon is explained by the fact that the corresponding groups are conjugate in $\GL_2(\hat{\Z})$. 
\end{remark}


 The other groups of Shimura type of genus 3 give rise to smooth plane quartics. In table \ref{tab: plane quartic genus 3} we present the plane quartics obtained.

\begin{table}[ht]
\begin{tabular}{|c|c|c|}\hline
    \bf{label}  &  \bf{name}    &   \bf{Plane quartic model}           \\ \hline
    7A3         &   $X(7)$      &   \small $  -xy^3 + x^3z + yz^3 = 0$     \\ 
    8A3         &               &   \small $ x^3z + 4xz^3-y^4 = 0 $        \\ 
    12O3        &               &   \small $ x^3z - 3x^2z^2 - xy^3 + 4xz^3 + 2y^3z - 2z^4 = 0$ \\ 
    15E3        &               &   \small $ x^3z - x^2 y^2 + xyz^2 - y^3 z - 5z^4 = 0$ \\ 
    16H3        &               &   \small $ -y^4 + x^3z + 4xz^3 = 0 $     \\
    20S3        &   $X_1(20)$   &   \small $ x^3z -x^2y^2 - 3x^2z^2 + xy^3 + 4xz^3 - 2z^4 = 0$  \\ 
    24X3        &               &   \small $ \begin{array}{c} x^3 z - 2x^2 yz - x^2z^2 - xy^3 + 2xy^2z \\
                                        + 6xyz^2 +2y^3z - 2y^2z^2 - 4xz^3 = 0 \end{array}$ \\ 
    24Y3        &               &   \small $ \begin{array}{c} x^3z - x^2y^2 - x^2z^2 + xz^3 - xy^2z \\
                                    - 3xyz^2 + y^3z + 2y^2z^2 + yz^3 = 0 \end{array}$ \\ 
    32J3        &               &   \small $ -y^4 + x^3z + 4xz^3 = 0$ \\ 
    34C3        &   $X_0(34)$   &   \small $ \begin{array}{c} -x^2y^2 + 2xy^3 - y^4 + x^3z + 3xy^2z + 4y^3z - \\
                                    3x^2z^2 - 3xyz^2 - 6y^2z^2 + 4xz^3 + 4yz^3 - 2z^4 = 0 \end{array} $ \\ 
    36K3        &               &   \small $ -xy^3 + x^3z + 2y^3z - 3x^2z^2 + 4xz^3 - 2z^4 = 0$ \\ 
    43A3        &   $X_0(43)$   &   \small $ \begin{array}{c} -2x^2y^2 + xy^3 - 9y^4 + x^3z + 2x^2yz + 3xy^2z + 24y^3z - \\
                                    2x^2z^2 - 5xyz^2 - 28y^2z^2 + 3xz^3 + 16yz^3 - 4z^4 = 0 \end{array}$ \\ 
    45D3        &   $X_0(45)$   &   \small $ -x^2y^2 + x^3z - y^3z + xyz^2 - 5z^4 = 0$ \\ 
    49A3        &               &   \small $ -xy^3 + x^3z + yz^3 = 0$ \\
    64B3        &   $X_0(64)$   &   \small $-y^4 + x^3z + 4xz^3 = 0$    \\
	\hline
\end{tabular}
\caption{Plane quartic Shimura type modular curves of genus $3$}
\label{tab: plane quartic genus 3}
\end{table}

\begin{remark}The curves corresponding to the groups labeled by 7A3 and 49A3 are isomorphic. The curves corresponding to the groups labeled by 8A3, 16H3, 32J3, 64B3 are isomorphic. The curves corresponding to the groups labeled by 12O3 and 36K3 are isomorphic. The curves corresponding to the groups labeled by 15E3 and 45D3 are isomorphic. No other curves in Table~\ref{tab: plane quartic genus 3} are isomorphic. 
Again, the explanation for these isomorphisms is that the congruence subgroups are conjugate in $\GL_2(\hat{\Z})$.
\end{remark}


We have also computed models for $92$ out of the $105$ congruence subgroups that are not of Shimura type and have $MK^2 \le 500$. An example of a plane quartic occurs for the group labeled $9A3$, cut out by the quartic $81x^4 - 54x^3y - 27x^2y^2 + 3xy^3 + y^4 - 729xz^3 + 486yz^3$.

\subsection{Genus 6}\label{subsec:g6}
Among the $8$ groups of Shimura type occurring, none admits a smooth plane model. We have also computed models for $19$ out of the $29$ congruence subgroups that are not of Shimura type and have $MK^2 \le 500$. 

Except the curve corresponding to 18A6, all the other ones are of genus $6$ and have a canonical model generated by quadrics, which means that they are non-hyperelliptic and that they do not admit a smooth plane model. 


For the curve 18A6, we also need cubic equations to define its canonical model. According to Theorem~\ref{thm:NEP}, this implies that it does admit a smooth plane model or that it is a trigonal curve. 
We first found an explicit birational equivalence between $\mathbb{P}^2$ and the locus of the quadrics in the ideal generating the canonical model. 
This birational map gives a parametrization of the locus of the quadrics. We plugged it into the degree 3 equations, and, after a suitable scaling of the variables, we found the following equation:
$$
y^3=(x - 3)(x + 1)(x^2 + 3)(x + 3)^2(x^2 + 6x + 21)^2.
$$

Interestingly, the curve (that it is not a smooth plane quintic) is not only trigonal, but also superelliptic. Notice that this is a quite remarkable exception since the dimensions of the moduli space of curves of genus $6$, and the locus of trigonal, plane and superelliptic ones of degree 3 inside it are: $\dim(\mathcal{M}_6)=15$, $\dim(\mathcal{M}^{trig}_6)=13$, $\dim(\mathcal{M}_6^{plane})=12$ and  $\dim(\mathcal{M}_6^{superell,3})=5$.
Actually, we could have guessed the existence of an automorphism of order $3$ of the curve, since $\Aut(G)$ contains an element $\sigma$ of order $3$, which induces an automorphism of the curve $X_G$. The induced action of the automorphism on the $q$-expansions is via $\sigma(f)(q) = f(\zeta_3 q)$. Therefore, we can readily compute its action on the curve, and observe that it corresponds to the action $x \mapsto x$ and $y \mapsto \zeta_3 y$ in the model we found.



\subsection{Genus 10} 

We have checked all the $17$ groups of Shimura type of genus $10$. None of them admits a smooth plane model. This was verified by computing the graded Betti number $\beta_{2,4}$. In $15$ out of the $17$ cases we have $\beta_{2,4} = 0$. Therefore, by Theorem~ \ref{thm:Koszul}, these curves do not admit a $g_6^2$, or equivalently a smooth plane model. The remaining cases, of the groups 46A10 and 92A10, are both isomorphic to the curve $X_0(92)$. In this case, we obtain $\beta_{2,4} = 27$, which by \cite{schreyer}*{Corollary 4.2} implies either that the curve is a smooth plane curve or that it is a double cover of an elliptic curve. However, in this case one checks that the quotient of the curve by the Atkin-Lehner involution $W_{23}$ yields an elliptic curve, hence by Lemma~\ref{lem:hyper} it does not admit a smooth plane model.
    
In most cases, projecting to $\P^2$ using the three divisors of maximal valuation at the cusp at $\infty$, which is a flex, one obtains a singular curve of degree $10$ or $11$ with coefficients of large height. However, for 
the groups in Table~\ref{tab: Shimura type modular curves of genus 10 which admit a morphism to an elliptic curve} we obtain a smooth cubic (an elliptic curve).


\begin{table}[ht]
\begin{tabular}{|c|c|c|}\hline
    \bf{label}  &  \bf{name}    &    \bf{Image in $\P^2$ }          \\ \hline
    9A10        &   $X(9)$      &   \small $y^2 = x^3 + 16$ \\ 
    18E10       &               &   \small $y^2 = x^3 - 12x^2 + 48x$ \\ 
    27B10       &               &   \small $y^2 = x^3 + 16$ \\ 
	36Q10 	    &               &   \small $y^2 = x^3 - 12x^2 + 48x$ \\ 
	54A10       &               &   \small $y^2 = x^3 - 12x^2 + 48x$ \\ 
	81A10       &               &   \small $y^2 = x^3 + 16$ \\ 
	108F10        &   $X_0(108)$   &   \small $y^2 = x^3 - 12x^2 + 48x$ \\ 
	\hline
\end{tabular}
\caption{Elliptic curves admitting a morphism from a Shimura type modular curve of genus $10$}
\label{tab: Shimura type modular curves of genus 10 which admit a morphism to an elliptic curve}
\end{table}

\subsection{Genus 15}
We have not been able to check the existence of a smooth plane model for any of the $23$ groups of Shimura type of genus $15$ using Algorithm~\ref{alg:cantoplane}, as computing the corresponding Betti numbers $\beta_{3,5}$ turned out to be beyond our computational ability. 

However, we can rule out the existence of smooth plane models for all these curves by looking at their Atkin-Lehner quotients. Indeed, any congruence subgroup $\Gamma(H,1)$ of Shimura type of level $N$ and parameter $t = 1$ satisfies $\Gamma_1(N) \subseteq \Gamma \subseteq \Gamma_0(N)$, and a subset of the Atkin-Lehner operators on $\Gamma_0(N)$ normalize $\Gamma$ as well. These induce automorphisms of the curve $X_G$ (for $G = G(H,1)$), and we denote the quotients by a subset $W$ of them by $X_G / W$. Whenever $|W| = 4$ and $X_G / W$ is hyperelliptic, we have a morphism $X_G \to \P^1$ of degree $8$, and we can deduce from Theorem~\ref{thm:Greco} that $X_G$ does not admit a smooth plane model (which must be of degree $d = 7$).

Moreover, each of the congruence subgroups of Shimura type of genus $15$ is isomorphic to a group of type $\Gamma(H,1)$. Table~\ref{tab: genus 15} summarizes our findings, where for groups which are conjugate in $\GL_2(\widehat{\Z})$ we have written down both labels.

\begin{table}[ht]
\begin{tabular}{|c|c|c|c|}\hline
    \bf{label}  &  \bf{name}    &    \bf{$W$} & \bf{$g(X_G / W)$ }           \\ \hline
    35C15, 175A15        &   $X_0(175)$      &  $\langle w_{25},w_7 \rangle$ & 3     \\
    40W15, 80R15       &               & $\langle w_{16},w_5 \rangle$ & 1 \\
    40X15, 80T15       &               & $\langle w_{16},w_5 \rangle$ &   1     \\
	43A15 	    &               & $\langle w_{43} \rangle$ &  7 \\
	51A15, 153A15       & $X_0(153)$              &  $ \langle w_{9},w_{17} \rangle$ & 2 \\
	60AC15       &               & $\langle w_3,w_5 \rangle$ &     1    \\
	60AD15        &      & $\langle w_3,w_5 \rangle$ & 1 \\
	67A15        &      & $ \langle w_{67} \rangle$ & 7 \\
	68D15, 136D15       &  $X_0(136)$    & $ \langle w_8,w_{17} \rangle$ & 3 \\
	85A15        &      & $ \langle w_5,w_{17} \rangle$ & 2 \\
	85B15        &      & $\langle w_5,w_{17} \rangle$ & 2 \\
	102C15        &  $X_0(102)$    & $\langle w_2,w_3, w_{17} \rangle$ & 1 \\
	110A15        &  $X_0(110)$    & $ \langle w_5, w_{11} \rangle$ & 1 \\
	141D15        &  $X_0(141)$    & $ \langle w_3, w_{47} \rangle$ & 1 \\
	155A15        &  $X_0(155)$    & $ \langle w_5, w_{31} \rangle$ & 1 \\
	161A15        &  $X_0(161)$    & $ \langle w_7, w_{23} \rangle$ & 2 \\
	179A15        &  $X_0(179)$    & $ \langle w_{179} \rangle$ &  3 \\
	193A15        &  $X_0(193)$    & $ \langle w_{193} \rangle$ & 7 \\
	\hline
\end{tabular}
\caption{Atkin-Lehner quotients of modular curves of Shimura type of genus $15$}
\label{tab: genus 15}
\end{table}

Note that, with the exception of $X_0(102), X_0(136), X_0(175), X_0(179), X_0(193)$ and the two curves corresponding to labels 43A15 and 67A15, the genus of all Atkin-Lehner quotients is either $1$ or $2$, implying that they are hyperelliptic, and $|W|=4$. By \cite{H1997} we deduce that $X_0^*(136) = X_0(136) / W$ is also hyperelliptic and thus we are left with six curves for which we were not able to rule out the existence of a smooth plane model using Theorem~\ref{thm:Greco}, namely $X_0(102), X_0(175), X_0(179), X_0(193)$ and the two curves corresponding to labels 43A15 and 67A15. Note that by \cite{FH1999} there is no quotient of either of $X_0(175), X_0(179), X_0(193)$ which is hyperelliptic.

In order to rule out the six remaining curves, we use Theorem~\ref{thm:Harui} with a simple application of Riemann-Hurwitz as follows. If $X_G$ admits a smooth plane model (of degree $d = 7$), and $w \in W$ is any Atkin-Lehner involution, by Theorem~\ref{thm:Harui} $w$ has $8$ fixed points. Riemann-Hurwitz then implies that the genus of the quotient is $g(X_G/ \langle w \rangle ) = 6$. Therefore, if we find some $w \in W$ where the genus of the quotient is not $6$, it does not admit a smooth plane model. Looking again at Table~\ref{tab: genus 15} we see that this rules out $X_0(179), X_0(193)$ and the two curves corresponding to labels 43A15 and 67A15. Finally, for $X_0(102)$ we see that $g(X_0(102) / \langle w_2 \rangle) = 7$, and for $X_0(175)$ we have $g(X_0(175) / \langle w_7 \rangle) = 8$, showing that they also do not admit a smooth plane model.

The map to $\P^2$ obtained by using the cusp at $\infty$ as the flex point always yields a singular curve of degree $16$, except the following cases.
For the group 60AC15, we obtain the elliptic curve 
\begin{equation*}
    y^2 = x^3 + 4x^2 - 16x.
\end{equation*}

For the groups 40W15 and 80R15  (which induce isomorphic curves) we obtain a septic with $3$ singular points, namely
\begin{align*}
    x^5y^2 &- x^6z - 5x^4y^2z + x^2y^4z + 4x^5z^2 + 12x^3y^2z^2 - 4xy^4z^2 - 6x^4z^3 \\
&- 16x^2y^2z^3 + 4y^4z^3 + 4x^3z^4 + 12xy^2z^4 - x^2z^5 - 4y^2z^5 = 0.
\end{align*}
\vspace{2ex}

\subsection{Genus 21}\label{genus21}
We have not been able to check the existence of a smooth plane model for any of the $55$ groups of Shimura type of genus $21$ using Algorithm~\ref{alg:cantoplane}, as computing the corresponding Betti numbers $\beta_{4,6}$ turned out to be beyond our computational ability. 

However, we can rule out the existence of smooth plane models for all but five of these curves by looking at their Atkin-Lehner quotients, and using Theorem~\ref{thm:Harui} and Riemann-Hurwitz as before. In this case, if $X_G$ admits a smooth plane model (of degree $d = 8$), and $w \in W$ is any Atkin-Lehner involution, by Theorem~\ref{thm:Harui} $w$ has $8$ fixed points. Riemann-Hurwitz then implies that the genus of the quotient is $g(X_G/ \langle w \rangle ) = 9$.
Table~\ref{tab: genus 21} shows, for each of these curves, the chosen Atkin-Lehner involution and the genus of the corresponding quotient, where for groups which are conjugate in $\GL_2(\widehat{\Z})$ we have written down all labels in a single line.

The curves left after this analysis are $X_0(256)$ and the curves corresponding to labels 41A21, 91A21, 91B21, 137A21.
The modular curve $X_0(256)$ has a single Atkin-Lehner involution, and its quotient by that involution has genus $9$. By Theorem \ref{thm:Harui}, if this curve admitted a smooth plane model, the corresponding quotient would have gonality $4$. However, by computing the Betti table of the quotient, we see that the $a(X_0(256)/\langle w \rangle)$ invariant of the table is $3$. Then, Corollary 9.7 in \cite{Eisenbud} implies that $\operatorname{Cliff}(X_0(256)/\langle w \rangle)\geq4$, so the gonality of  $X_0(256)/\langle w \rangle$ is at least $6$. On the other hand, it is also bounded by $6$ because of the genus being $9$. This proves that $X_0(256)$ does not admit a smooth plane model. For the other four curves, at the present time, no algorithm is implemented to compute their Atkin-Lehner quotients.  

\begin{table}[ht]
\begin{tabular}{|c|c|c|c|}\hline
    \bf{label}                              &  \bf{name}        &    \bf{$W' \subseteq W$}          & \bf{$g(X_G / W')$ }   \\ \hline
    21C21, 147B21                           &                   & $\langle w_{3} \rangle$           & 10                    \\
    24A21, 48CE21, 96BA21, 192P21           &  $X_0(192)$       & $\langle w_{64} \rangle$          & 11                    \\
	28M21, 56L21, 112E21 	                &                   & $\langle w_{112} \rangle$         & 10                    \\
	30G21, 90O21                            &                   &  $ \langle w_{10} \rangle$        & 11                    \\
	33C21                                   &  $X_1(33)$        & $\langle w_{33} \rangle$          & 10                    \\
	34A21                                   &  $X_1(34)$        & $\langle w_{34} \rangle$          & 10                    \\
	35B21, 245A21                           &  $X_0(245)$       & $ \langle w_{245} \rangle$        & 8                     \\
	36F21, 72AC21                           &                   & $ \langle w_{72} \rangle$         & 10                    \\
	39B21, 117A21                           &                   & $ \langle w_{13} \rangle$         & 11                    \\
	42H21, 84Q21                            &                   & $\langle w_3 \rangle$             & 10                    \\
	45F21                                   &                   & $ \langle w_{45} \rangle$         & 10                    \\
	52C21, 104B21                           &                   & $ \langle w_{13} \rangle$         & 11                    \\
	55B21                                   &                   & $ \langle w_5 \rangle$            & 3                     \\
	56M21                                   &                   & $ \langle w_8 \rangle$            & 6                     \\
	65A21                                   &                   & $ \langle w_{13} \rangle$         & 6                     \\
	69A21, 207A21                           &  $X_0(207)$       & $ \langle w_{23} \rangle$         & 8                     \\
	72AC21                                  &                   & $ \langle w_{72} \rangle$         & 10                    \\
	78C21                                   &                   & $ \langle w_{26} \rangle$         & 8                     \\
	92A21, 184A21                           &  $X_0(184)$       & $ \langle w_{23} \rangle$         & 5                     \\
	97A21                                   &                   & $ \langle w_{97} \rangle$         & 10                    \\
	111A21                                  &                   & $ \langle w_{111} \rangle$        & 7                     \\
	115A21                                  &                   & $ \langle w_{23} \rangle$         & 6                     \\
	119A21                                  &                   & $ \langle w_{7} \rangle$          & 6                     \\
	133B21                                  &                   & $ \langle w_{133} \rangle$        & 10                    \\
	138A21                                  &  $X_0(138)$       & $ \langle w_{23} \rangle$         & 5                     \\
	154B21                                  &  $X_0(154)$       & $ \langle w_{22} \rangle$         & 11                    \\
	165B21                                  &  $X_0(165)$       & $ \langle w_{3} \rangle$          & 11                    \\
	178A21                                  &  $X_0(178)$       & $ \langle w_{89} \rangle$         & 8                     \\
	201A21                                  &  $X_0(201)$       & $ \langle w_{67} \rangle$         & 11                    \\
	215A21                                  &  $X_0(215)$       & $ \langle w_{43} \rangle$         & 11                    \\
	247A21                                  &  $X_0(247)$       & $ \langle w_{247} \rangle$        & 8                     \\
 	251A21                                  &  $X_0(251)$       & $ \langle w_{251} \rangle$        & 4                     \\
 	257A21                                  &  $X_0(257)$       & $ \langle w_{257} \rangle$        & 7                     \\
	\hline
\end{tabular}
\caption{Atkin-Lehner quotients of modular curves of Shimura type of genus $21$ indicating no smooth plane model}
\label{tab: genus 21}
\end{table}

\subsection{Proof  of Theorem \ref{thm:main}}
We are ready to prove now our main theorem:

\begin{proof} 
The first claim is Theorem \ref{finite} and the second one is deduced from its proof. The last claim is a consequence of the computations in this section.
\end{proof}



\bibliographystyle{alphaabbr}
\bibliography{synthbib}

\end{document}